\documentclass[12pt,draftcls,onecolumn]{IEEEtran}
\IEEEoverridecommandlockouts
\usepackage{cite}
\usepackage{graphicx}      

\usepackage[utf8]{inputenc}
\usepackage[T1]{fontenc}

\newtheorem{theorem}{\bf Theorem}

\newtheorem{lemma}{\bf Lemma}
\newtheorem{corollary}{\bf Corollary}

\newtheorem{remark}{\bf Remark}
\newtheorem{claim}{\bf Claim}

\newcommand{\vnorm}[1]{\left|\left|#1\right|\right|}

\usepackage{color}
\definecolor{Red}{rgb}{1,0,0}
\definecolor{Blue}{rgb}{0,0,1}
\definecolor{Green}{rgb}{0,1,0}
\definecolor{magenta}{rgb}{1,0,.6}
\definecolor{lightblue}{rgb}{0,.5,1}
\definecolor{lightpurple}{rgb}{.6,.4,1}
\definecolor{gold}{rgb}{.6,.5,0}
\definecolor{orange}{rgb}{1,0.4,0}
\definecolor{hotpink}{rgb}{1,0,0.5}
\definecolor{newcolor2}{rgb}{.5,.3,.5}
\definecolor{newcolor}{rgb}{0,.3,1}
\definecolor{newcolor3}{rgb}{1,0,.35}
\definecolor{darkgreen1}{rgb}{0, .35, 0}
\definecolor{darkgreen}{rgb}{0, .6, 0}
\definecolor{darkred}{rgb}{.75,0,0}

\newcommand{\bigO}[1]{\mathcal{O} \left(#1 \right)}
\newcommand{\ts}{t^\star}
\newcommand{\ds}{\Delta s}

\usepackage{amsmath}   

\usepackage{amssymb}
\usepackage{algorithm}
\usepackage{algorithmic}
\usepackage{array}
\usepackage{comment}
\usepackage{url}

\begin{document}

\title{Robust Distributed Averaging in Networks}

\author{\IEEEauthorblockN{Ali Khanafer\IEEEauthorrefmark{1}, Behrouz Touri\IEEEauthorrefmark{2}, and Tamer Ba\c{s}ar\IEEEauthorrefmark{1}}

\IEEEauthorblockA{\IEEEauthorrefmark{1}Coordinated Science Laboratory, University of Illinois at Urbana-Champaign, USA, Email:\url{{khanafe2,basar1}@illinois.edu}}
\IEEEauthorblockA{\IEEEauthorrefmark{2}ECE Department, Georgia Institue of Technology, Atlanta, GA, USA, Email: \url{touri@gatech.edu}}
\thanks{*This work was supported in part by an AFOSR MURI Grant FA9550-10-1-0573.}
}

\maketitle

\begin{abstract}
In this work, we consider two types of adversarial attacks on a network of nodes seeking to reach consensus. The first type involves an adversary that is capable of breaking a specific number of links at each time instant. In the second attack, the adversary is capable of corrupting the values of the nodes by adding a noise signal. In this latter case, we assume that the adversary is constrained by a power budget. We consider the optimization problem of the adversary and fully characterize its optimum strategy for each scenario.
\end{abstract}


\section{Introduction}
Starting with the work of \cite{TsitsiklisThesis}, distributed computation received increased attention. The core idea behind various distributed decision applications is the ability of individual agents to reach agreement \emph{globally} via \emph{local} interactions. Fields where this idea is key  include flocking and multiagent coordination \cite{BlondelSurvey,SaberFlocking,JadbabaieNNRules}, optimization \cite{LiBasar,NedicOzdaglarParrilo}, and the study of social influence \cite{JacksonGolub}.

As a particular example, consensus averaging involves agents who attempt to converge to the average of their initial measurements through local averaging. Consensus problems can be formulated in both discrete and continuous time. Consensus protocols find many applications in sensor networks where sensors collaborate distributively to make measurements of a certain quantity, such as the temperature in a field. The convergence of consensus algorithms has been studied widely; see \cite{BlondelSurvey}.

The convergence of consensus protocols under the effect of non-idealities have also been studied in the literature. \cite{KashyapBasarSrikant} study the convergence properties of pairwise gossip under the constraint that agents can only store integer values. Consensus in networks with noisy links was explored by \cite{XiaoBoyd} and \cite{TouriNedic}. \cite{SarwateDimakis} consider the case where the nodes are allowed to be mobile. The effects of switching topologies and time delays were considered in \cite{SaberMurray}, \cite{NedicOzdaglar}, and \cite{TouriNedicTAC1,TouriNedicArXiv,TouriNedicCDC}.

Here, we study the problem of continuous-time consensus averaging in the presence of an intelligent adversary. We consider two network-wide attacks launched by an adversary attempting to hinder the convergence of the nodes to consensus. The adversarial attacks we explore here differ from the ones studied by \cite{Sundaram}, \cite{BicchiBullo}, and \cite{SandbergJohansson}, who consider the effect of malicious and compromised agents who could update their values arbitrarily. In the first scenario (called \textsc{Attack-I}) we consider, the adversary can break a set of edges in the network at each time instant. In practice, the adversary would be limited in its resources; we translate this practical limitation to a hard constraint on the total number of links the adversary can compromise at each time instant. In the second case (called \textsc{Attack-II}), the adversary can corrupt the measurements of the nodes by injecting a signal under a maximum power constraint. Our goal is to study the optimal behavior of the adversary in each case, given the imposed constraints.

For both attacks, we formulate the problem of the adversary as a finite horizon maximization problem in which the adversary seeks to maximize the Euclidean distance between the nodes' state and the consensus line. We use Pontryagin's maximum principle (MP) to completely characterize the optimal strategy of the adversary under both attacks;  for each case we obtain a closed-form solution, providing also a potential-theoretic interpretation of the adversary's optimal strategy in \textsc{Attack-I}. Furthermore, we support our findings with numerical studies.

In Sect~\ref{AttackI}, we describe \textsc{Attack-I} and formulate and solve the adversary's problem. We study \textsc{Attack-II} in Sect~\ref{AttackII}, present simulation results in Sect~\ref{Simulations}, and conclude in Sect~\ref{Conclusion}.

\section{Attack I: An Adversary Capable of Breaking Links} \label{AttackI}
Consider a connected network of $n$ nodes described by an undirected graph  $\mathcal{G} = (\mathcal{N},\mathcal{E})$, where $\mathcal{N}$ is the set of vertices, and $\mathcal{E}$ is the set of edges with $|\mathcal{E}| = m$. The nodes of the network are the vertices of $\mathcal{G}$, i.e., $|\mathcal{N}| = n$; we will denote an edge in $\mathcal{E}$ between nodes $i$ and $j$ by $(i,j)$. The value, or state, of the nodes at time instant $t$ is given by $x(t) = [x_1(t),...,x_n(t)]^T$. The nodes start with an initial value $x(0)=x_0$, and they are interested in computing the average of their initial measurements $x_{avg} = \frac{1}{n}\sum_{i=1}^n x_i(0)$ via local averaging. We consider the continuous-time averaging dynamics given by
\[
\dot{x}(t) = A(t)x(t), \quad x(0) = x_0,
\]
where the rows of the matrix $A(t)$ sum to zero and its off diagonal elements are nonnegative. We also assume that $A(t)$ is symmetric. Further, we define $\bar{x} = \mathbf{1}x_{avg}$ and let $M = \frac{\mathbf{1}\mathbf{1}^T}{n}$. A well known result states that, given the above assumptions, the nodes will reach consensus, i.e.,  $\lim_{t\to \infty} x(t) =\bar{x}$. 

An adversary attempts to slow down convergence. At each time instant, he can break at most $\ell \leq m$ links and wishes to select the links which will cause the most harm. Let $u_{ij}(t) \in \{0,1\}$ be the weight the adversary assigns to link $(i,j)$. He breaks link $(i,j)$ at time $t$ when $u_{ij}(t) = 1$. His control is given by $u(t)=[u_{12}(t),u_{13}(t),...,u_{1n}(t),u_{23}(t),...,u_{(n-1)n}(t)]^T$. If $(i,j) \notin \mathcal{E}$, then $u_{ij}(t) = 0$ for all $t$. We will denote the number of links the adversary breaks at time $t$ by $N_u(t)$. Then, $N_u(t) = |u(t)|^2$, where $|.|$ denotes the Euclidean norm.  In accordance with the above, the strategy space of the adversary is
\[
\mathcal{U} = \left\{u(t) \in \{0,1\}^{n\choose2}: N_u(t) \leq \ell \right\}
\]. 

The objective function of the adversary is:
\begin{equation*}
J(u) = \int_0^T k(t)\left| x(t)-\bar{x}\right|^2 dt,
\end{equation*}
where the kernel $k(t)$ is positive and integrable over $[0,T]$. The adversary's problem can be formulated as follows:
\begin{eqnarray}
& \max_{u(t)\in \mathcal{U}} & \quad J(u) \label{prb::JammerPrb} \\
& \text{s.t.} & \dot{x}(t) = A(t)x(t),\quad A_{ii}(t) =  - \sum_{j=1,j\neq i}^n A_{ij}(t),\nonumber \\
&&A_{ij}(t) = A_{ji}(t) = a_{ij}\left(1-u_{ij}(t)\right),\nonumber\\
&& a_{ij} \geq 0, a_{ij} >0 \Leftrightarrow (i,j) \in \mathcal{E}. \nonumber
\end{eqnarray}
The Hamiltonian is then given by:
\[
H(x,p,u) = k(t)\left|x(t)-\bar{x}\right|^2 + p(t)^TA(t)x(t).
\]
The first-order necessary conditions for optimality are:
\begin{eqnarray}
\dot{p}(t) & = & -\frac{\partial }{\partial x}H(x,p,u) \nonumber \\
& = & -2k(t)(x(t)-\bar{x}) - A(t)^Tp(t), \quad p(T) = 0 \label{eqn::ODEp1} \\
\dot{x}(t) & = & A(t)x(t), \quad x(0) = x_0 \label{eqn::ODEx1} \\
u^\star & = & \arg \max\left\{H(x,u,p,t) : u \in \mathcal{U}\right\}, \nonumber
\end{eqnarray}
where $x(t),p(t) \in C^1[0,T]$, the space of continuously differentiable functions over $[0,T]$.
Let $\Phi_x(t,0)$ and $\Phi_p(t,0)$ be the state transition matrices corresponding to $x(t)$ and $p(t)$, respectively. Then, the solutions to ODEs (\ref{eqn::ODEp1}) and (\ref{eqn::ODEx1}) are:
\begin{eqnarray}
x(t) & = &  \Phi_x(t,0)x_0  \label{eqn::xDyn}, \\
p(t) & = & \Phi_p(t,0)p(0) - 2\int_0^t [\Phi_p(t,\tau)k(\tau)(x(\tau)-\bar{x})]d\tau. \nonumber
\end{eqnarray}
Using the terminal condition $p(T) = 0$, we can write:
\[
p(0) = 2\int_0^T \Phi_p(0,\tau)k(\tau)\left[\Phi_x(\tau,0)x_0-\bar{x}\right]d\tau.
\]
We therefore have
\begin{equation}
 p(t) = 2\int_t^T \Phi_p(t,\tau)k(\tau)\left[\Phi_x(\tau,0)x_0-\bar{x}\right]d\tau. \label{eqn::pDyn}
\end{equation}
Let us write
\begin{eqnarray*}
&&\hspace{-3.5mm} p(t)^TA(t)x(t) = \sum_{i=1}^n p_i(t) \left(\sum_{j=1}^n A_{ij}(t)x_j(t) \right) \\
&&\hspace{-3.5mm} =  \sum_{i=1}^n p_i(t) \left(- \sum_{j=1,j\neq i}^n a_{ij}u_{ij}(t)x_i(t) + \sum_{j=1, j\neq i}^n a_{ij}u_{ij}(t)x_j(t) \right) \\
&&\hspace{-3.5mm} = \sum_{i=1}^n \sum_{j=1, j\neq i}^n u_{ij}(t)a_{ij}p_i(t)(x_j(t) - x_i(t)) \\
&&\hspace{-3.5mm} = \sum_{j=2}^{n}\sum_{i=1}^{j-1} u_{ij}(t)a_{ij}(p_j(t)-p_i(t))(x_i(t) - x_j(t)).\\
\end{eqnarray*}
We further have
\begin{eqnarray*}
&& \max_{u(t) \in \mathcal{U}} H(x,p,u) = \max_{u(t) \in \mathcal{U}} k(t)\left|x(t)-\bar{x}\right|^2 + p(t)^TA(t)x(t)\\
&& = k(t)\left|x(t)-\bar{x}\right|^2 + \sum_{j=2}^{n}\sum_{i=1}^{j-1} \max_{u_{ij}(t) \in \{0,1\}}u_{ij}(t)f_{ij}(A,x,p),
\end{eqnarray*}
where $f_{ij}(A,x,p) = a_{ij}(p_j(t)-p_i(t))(x_i(t) - x_j(t))$. Let $(f_1,...,f_m) = \pi(f)$ be a nondecreasing ordering of the $f_{ij}$'s. Define the subset of edges $\mathcal{\tilde{I}}_t$ as follows: $\tilde{\mathcal{I}}_t = \left\{(i,j) \in \mathcal{N}: f_{ij} < 0 \text{ and } f_{ij} \leq f_{\ell+1} \right\}$. Further, let $\mathcal{I}_t$ be the set containing the $\ell$ elements of $\mathcal{\tilde{I}}_t$ having the smallest values. Hence, we conclude that the optimal control is:
\begin{eqnarray}
u^\star_{ij}(t) =  \left\{
  \begin{array}{l l}
    1 & \quad \text{if $(i,j) \in \mathcal{I}_t$}  \\
    0 & \quad \text{if $f_{ij} > 0$} \label{eqn::OptCtrlMP}\\
    \{0,1\} & \quad \text{if $f_{ij} = 0$} 
  \end{array} \right.
\end{eqnarray}
The functions $f_{ij}$ depend on both the state and the co-state, which in turn are defined in terms of the control. This makes it hard to obtain a closed-form solution for the control. However, in the following we consider the utility of the adversary which allows us to completely characterize his optimal strategy; we will be using the term "connected component" to refer to a set of connected nodes which have the same values. Let $w_{ij}(t) := a_{ij}(x_j(t) - x_i(t))^2$.
\begin{theorem} \label{thm::main}
For all $t$, the optimal strategy of the adversary, $u^\star(t)$, is to break $\ell$ links with the highest $w_{ij}(t)$ values. Furthermore, if the adversary has an optimal strategy of breaking less than $\ell$ links, then either $\mathcal{G}$ has a cut of size less than $\ell$ or the nodes have reached consensus at time $t$. In either of the cases, breaking $\ell$ links is also optimal.
\end{theorem}
\begin{proof}
We first characterize $N_{u^\star}(t)$, for all $t$. Because $x(t),p(t) \in C^1[0,T]$, the value of $f_{ij}$ cannot change abruptly in a finite interval. As a result, the control obtained from the MP cannot switch infinitely many times in a finite interval. To this end, let $[s,s+\Delta s]$, $\Delta s > 0$, be a small subinterval of $[0,T]$ over which the adversary applies a stationary strategy $u^A$ such that $N_{u^A} < \ell$ with a corresponding system matrix $A$. Because the control strategy is time-invariant, the state trajectory is given by
\[
x(t) = e^{A(t-s)}x(s), \quad t \in [s,s+\Delta s].
\]
Let $P(t):=e^{At}$. Due to the structure of $A$, $P(t)$ is a doubly stochastic matrix for $t \geq 0$; see \cite{Norris}, p. $63$.

Note that we can write $x(s) = \tilde{P}x_0$, where $\tilde{P}$ is some doubly stochastic matrix. Indeed, assume that the control had switched once at time $\tilde{s} \in [0,s)$, and that the system matrix over $[0,\tilde{s})$ was $\tilde{A}_1$, and the system matrix corresponding to $[\tilde{s},s)$ was $\tilde{A}_2$. Then $x(s) = e^{\tilde{A}_2(s-\tilde{s})}e^{\tilde{A}_1\tilde{s}}x_0$. Because both $e^{\tilde{A}_1t}, e^{\tilde{A}_2t}$ are doubly stochastic matrices, their product is also doubly stochastic. We can readily generalize this result to any number of switches in the interval $[0,s)$. With this observation, we can write
\[
x(t)-\bar{x} =P(t-s)\tilde{P}x_0-Mx_0 = (P(t-s) - M)x(s),
\]
where the last equality follows from the fact that
\begin{equation} \label{prop::M}
\tilde{P}M = M\tilde{P} = M,\text{ $\tilde{P}$ is doubly stochastic}.
\end{equation}
Let $u^B$ be a strategy identical to $u^A$ except at link $(i,j)$, where $u^A_{ij} = 0$ and $u_{ij}^B = 1$. Let the matrix $B$ be the system matrix corresponding to $u^B$, and define the doubly stochastic matrix $Q(t) := e^{Bt}$, $t\geq 0$. It follows that:
\begin{equation}
A_{ij} > B_{ij}=0, \quad A_{kl} = B_{kl} \quad \forall \mathcal{E} \ni (k,l) \neq (i,j).  \label{eqn::AvsB}
\end{equation}
We want to show that switching to strategy $u^B$ at some time $t^\star \in [s,s+\Delta s]$, can improve the utility of the adversary. Formally, we want to prove the following inequality:
\begin{eqnarray*}
&&\int_s^{s+ \Delta s} k(t)\left|(P(t-s)-M)x(s)\right|^2dt  \\
&&< \int_s^{t^\star} k(t)\left|(P(t-s)-M)x(s)\right|^2dt \nonumber \\
&&+ \int_{t^\star}^{s+\Delta s} k(t)\left|(Q(t-t^\star)-M)P(t^\star - s)x(s)\right|^2dt, \nonumber \\
\end{eqnarray*}
or equivalently
\begin{eqnarray}
&&\int_{t^\star}^{s+\Delta s} k(t)\cdot \left[\left|(Q(t-t^\star)-M)P(t^\star-s)x(s)\right|^2\right. \nonumber\\
&&- \left. \left|(P(t-s)-M)x(s)\right|^2\right] dt > 0. \label{eqn::intermediate}
\end{eqnarray}
Using (\ref{prop::M}) and the semi-group property, (\ref{eqn::intermediate}) simplifies to
\begin{equation}
\int_{t^\star}^{s+\Delta s} k(t)\cdot x(s)^T\Lambda(t,t^\star)x(s) dt > 0, \label{ineq::toProve}
\end{equation}
where $\Lambda(t,t^\star)= P(t^\star-s)Q(2(t-t^\star))P(t^\star-s) - P(2(t-s))$.
A sufficient condition for (\ref{ineq::toProve}) to hold is
\begin{equation}
h(t,x(s)) = x(s)^T\Lambda(t,t^\star)x(s)>0, \text{ for } t> t^\star.
\end{equation}
As $t \downarrow 0$, we can write $P(t) = I + tA+\mathcal{O}\left(t^2\right)$, where $\mathcal{O}\left(t^2\right)/t \leq K$ for sufficiently small $t$ and some finite constant $K$. We therefore have
\begin{eqnarray*}
&& \Lambda(t,t^*) = \left(I+(t^\star-s) A+\mathcal{O}\left(t^2\right)\right)\left(I+2(t-t^\star)B \right.\\
&&\left.+\mathcal{O}\left(t^2\right)\right) \left(I+(t^\star-s) A+\mathcal{O}\left(t^2\right)\right) - \left(I+2(t-s)A \right.\\
&&\left. +\mathcal{O}\left(t^2\right)\right) = 2(t-t^\star)B + 2(t^\star-s)A - 2(t-s)A \\
&&+ \mathcal{O}\left(t^2\right) \nonumber = 2(t-t^\star)(B-A) + \mathcal{O}\left(t^2\right). \label{eqn::matApprox}
\end{eqnarray*}
For sufficiently small $t$ and $t^\star$, the first term dominates the second term. For any symmetric $L$ with $L\mathbf{1} = 0$, the quadratic form exhibits the following form: $x^TLx = -\sum_{l=1}^n \sum_{k=1}^{l-1} L_{kl}(x_l-x_k)^2$, for any $x \in \mathbb{R}^n$. Using (\ref{eqn::AvsB}), we can then write
\begin{eqnarray}
h(t,x(s)) & = & 2(t-t^\star) \sum_{l=1}^{n} \sum_{k=1}^{l-1} (A_{kl} - B_{kl})\left(x_l(s)-x_k(s)\right)^2 \nonumber \\
& = & 2(t-t^\star)A_{ij}\left(x_j(s)-x_i(s) \right)^2. \label{eqn::condInit}
\end{eqnarray}
Hence, if there is a link $(i,j)$ such that $x_i(s) \neq x_j(s)$, there exists $t^\star$, $\tilde{t}$ such that $h(t,x(s))>0$ for $t \in \left(t^\star,\tilde{t}\right]$. By the semi-group property, we can write
\[
P(t) = P\left(\frac{t}{r} + \cdots+ \frac{t}{r}\right) = P\left(\frac{t}{r}\right)^r, \quad \forall r \in \mathbb{N}.
\]
Thus, for any $t\geq 0$, not necessarily small, and by selecting $r$ to be sufficiently large, we have:
\[
P\left(\frac{t}{r}\right)^r = \left(I + \frac{t}{r}A + \mathcal{O}\left(\frac{t^2}{r^2}\right)\right)^r\approx I + tA + \mathcal{O}\left(\frac{t^2}{r^2}\right).
\]
By following the same analysis as above with this approximation, we conclude that we can always find $t^* \geq s$ such that $h(t,x(s)) > 0$ for $t>t^*$. Since $s$ was arbitrary, we conclude that the optimal strategy must satisfy $N_{u^\star}(t) = \ell$ for all $t$, given that each of the $\ell$ links connects two nodes having different values. If no such link exists at a given time $s$, the adversary does not need to break additional links, although breaking more links does not affect optimality because $h(t,x(s))=0$ in such case. There are two cases under which the adversary cannot find a link to make $h(t,x(s))>0$: (i) The graph at time $s$ is one connected component. In this case, the nodes have already reached consensus and $N_{u^\star} = 0 < \ell$. This is a \emph{losing strategy} for the adversary as it failed in preventing nodes from reaching agreement; (ii) The graph at time $s$ has multiple connected components, and the number of links connecting the components is less than $\ell$. The adversary here possesses a \emph{winning strategy} with $N_{u^\star} < \ell$, as it can disconnect $\mathcal{G}$ into multiple components and prevent consensus.

Thus far, we have shown that the adversary can improve his utility if he switches at some time $t^\star \in [s,s+\Delta s]$ from strategy $A$ to strategy $B$ (where strategy B corresponds to the proposed optimal control). Now, we want to show that switching to strategy $B$ guarantees an improved utility for the adversary regardless of how the original trajectory $A$ changes beyond time $s+\Delta s$. To show this, we will assume that from time $s+\Delta s$ onward, strategy $B$ will mimic the original strategy. Assume that strategy $A$ switches to another strategy $C$ (hence, strategy $B$ will also switch to strategy $C$). Let us denote the system matrix corresponding to strategy $C$ by $R(t) := e^{Ct}$. Let us also restrict our attention to a small interval $[s+\Delta s, s + 2\Delta s]$ over which we can assume that the system is time-invariant. 

We want to prove the following inequality:
\begin{eqnarray*}
\int_{s+\Delta s}^{s + 2\Delta s} k(t)\cdot  \left[ \underbrace{|(R(t-(s+\Delta s))-M )Q(s+\Delta s - t^\star)P(t^\star -s) x(s)|^2}_{:=I_1}  - \right.\\
\left. \underbrace{| (R(t-(s+\Delta s))-M )P(s + \Delta s -s) x(s)|^2}_{:=I_2}   \right] dt > 0.
\end{eqnarray*}
As before, it suffices to prove that the integrant (in particular, $I_1-I_2$) is positive. Let us now expand both $I_1$ and $I_2$.

\begin{eqnarray*}
I_1 & = & x(s)^TP(\ts -s)Q(s+\ds-\ts)(R(t-(s+\ds))-M)(R(t-(s+\ds))-M)\\
&&Q(s+\ds-\ts)P(\ts -s)x(s)\\
& = & x(s)^TP(\ts -s)Q(s+\ds-\ts)(R(2(t-(s+\ds)))-M)Q(s+\ds-\ts)P(\ts -s)x(s)\\
& = & x(s)^T(P(\ts -s)Q(s+\ds-\ts)R(2(t-(s+\ds)))Q(s+\ds-\ts)P(\ts -s)-M)x(s).
\end{eqnarray*}
Similarly,
\[
I_2 = x(s)^T(P(\ds)R(2(t-(s+\ds)))P(\ds)-M)x(s) 
\]
Further, we have
\begin{eqnarray*}
I_1 - I_2 = x(s)^T(\underbrace{P(\ts -s)Q(s+\ds-\ts)R(2(t-(s+\ds)))Q(s+\ds-\ts)P(\ts -s)}_{F_1}\\
 - \underbrace{P(\ds)R(2(t-(s+\ds)))P(\ds)}_{F_2})x(s).
\end{eqnarray*}
Before we perform a first-order Taylor expansion to the above terms, let us define the following quantities:
$$
\tau_1 = t^\star -s, \quad \tau_2 = (s+\Delta s) - t^\star, \quad \tau_3 = t - (s+\Delta s),
$$
where $t^\star \in [s,s+\Delta s]$ and $t \in [s+\Delta s, s+2\Delta s]$.
\begin{figure}[!t]
\centering
\includegraphics[width=9.5cm]{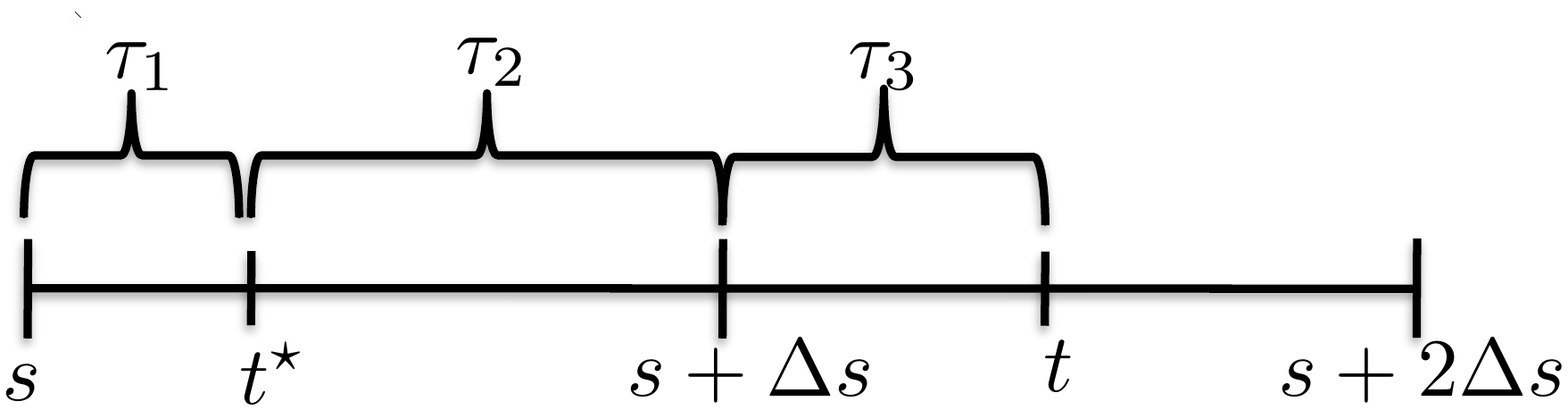}
\label{fig::TwoIntProof}
\end{figure}

Recall that we write $f(x) = \bigO{g(x)}$ as $x\to a$ if $\exists$ constants $M,\delta$ such that
\[
|f(x)| \leq M|g(x)|, \quad \text{for all $x$ satisfying $|x-a| < \delta$.}
\] 
Also, recall that following properties:
\begin{itemize}
\item $f(x)\bigO{g(x)} = \bigO{f(x)g(x)}$
\item $c\cdot \bigO{f(x)} = \bigO{f(x)}$, $c$ is a constant  
\end{itemize}
Using the above definition, we can prove the following claims.
\begin{claim}
As $\ds \to 0$, we have:
\begin{itemize}
\item[1-] If $f(\tau_i,\Delta s) = \bigO{\tau_i^2}$, then $f(\tau_i,\Delta s) = \bigO{\Delta s^2},$ $i \in\{1,2,3\}$ 
\item[2-]  If $f(\tau_i,\tau_j,\Delta s) = \tau_i\bigO{\tau_j^2}$, then $f(\tau_i,\tau_i,\Delta s) = \bigO{\ds^3},$ $i,j \in\{1,2,3\}$ 
\end{itemize}
\end{claim}
\begin{proof}
We proceed by using the above definition and properties.
\begin{itemize}
\item[1-] We have that $f(\tau_i,\Delta s) \leq M\tau_i^2 \leq M \ds^2$. Hence, $f(\tau_i,\Delta s) = \bigO{\ds^s}$.
\item[2-]  $f(\tau_i,\tau_j,\Delta s) = \tau_i\bigO{\tau_j^2} = \bigO{\tau_i \tau_j^2}$. Hence, $f(\tau_i,\tau_j,\Delta s) \leq M\tau_i \tau_j^2\leq M\ds^3$ and $f(\tau_i,\tau_j,\Delta s) = \bigO{\ds^3}$. 
\end{itemize}
\end{proof}
We can now expand $F_1$ and $F_2$ as follows. Note that, $\ds \to 0$, $f+\bigO{\ds^2}+\bigO{\ds^3}=f+\bigO{\ds^2}$.
\begin{eqnarray*}
F_1 & = & \left(I+\tau_1A+\bigO{\tau_1^2}\right)\left(I+\tau_2B+\bigO{\tau_2^2}\right)\left(I+2\tau_3C+\bigO{\tau_3^2}\right)\left(I+\tau_2B+\bigO{\tau_2^2}\right)\\
&&\left(I+\tau_1A+\bigO{\tau_1^2}\right) \\
& = & \left(I+\tau_1A+\tau_2B+\bigO{\ds^2}\right)\left(I+2\tau_3C+\bigO{\ds^2}\right)\left(I+\tau_1A+\tau_2B+\bigO{\ds^2}\right)\\
& = & \left(I+\tau_1A+\tau_2B+2\tau_3C+\bigO{\ds^2}\right)\left(I+\tau_1A+\tau_2B+\bigO{\ds^2}\right)\\
& = & I + 2\tau_1A + 2\tau_2B + 2\tau_3C + \bigO{\ds^2}.\\ \\
F_2 & = & \left(I+\ds A + \bigO{\ds^2}\right)\left(I +2\tau_3C+\bigO{\tau_3^2}\right)\left(I+\ds A + \bigO{\ds^2}\right) \\
& = &  \left(I+\ds A + 2\tau_3C+\bigO{\ds^2}\right)\left(I+\ds A + \bigO{\ds^2}\right)\\
& = & I+2\ds A + 2\tau_3C + \bigO{\ds^2}.
\end{eqnarray*}

Hence, we have
\begin{eqnarray*}
F_1 - F_2 & = & 2\left(\tau_1-\ds\right)A+2\tau_2B+\bigO{\ds^2}, \\
& = & 2\tau_2\left(B -A\right) +\bigO{\ds^2}\\
& = & 2\left(\left(s+\ds\right)-\ts\right)\left(B-A\right) +\bigO{\ds^2}
\end{eqnarray*}
and therefore, using (8), we obtain
\[
I_1 - I_2 = 2\left(s+\ds -\ts\right)A_{ij}\left(x_j(s)-x_i(s)\right)^2 > 0,
\]
as required.

It remains to show that the links the adversary breaks have the highest $w_{ij}(t)$ values. Let us again restrict our attention to the interval $[s,s+\Delta s]$ where the adversary applies strategy $u^A$. Assume (to the contrary) that the links the adversary breaks over this interval are not the ones with the highest $w_{ij}(t)$ values. In particular, assume that the adversary chooses to break link $(k,l)$, while there is a link $(i,j)$ such that $w_{ij}(t) > w_{kl}(t)$, $t \in [s,s+\Delta s]$. Assume that the adversary switches at time $t^\star \in [s,s+\Delta s]$ to strategy $u^B$ by \emph{breaking} link $(i,j)$ and \emph{unbreaking} link $(k,l)$. Then, (\ref{eqn::condInit}) becomes $h(t,x(s))=2(t-t^*)\left(w_{ij}(s)-w_{lk}(s)  \right)$. Hence, by following the same arguments as above, we conclude that breaking $(k,l)$ is not optimal. The proof is thus complete. 
\end{proof}
\begin{remark}
Potential theory aids in providing an interpretation of the result of Thm \ref{thm::main}. Consider an electrical network with $x_i$ being the voltage at node $i$ with respect to a fixed ground reference. Then, $(x_j-x_i)$ represents the potential difference (or voltage) $V$ across link $(i,j)$. The edge weight $a_{ij}$ represents the conductance of the link. It follows that the weight $w_{ij}$ is in fact the power $P$ dissipated across link $(i,j)$. Hence, the adversary will break the links with the highest power dissipation.
\end{remark}

\begin{theorem}
The optimal strategy derived in Thm \ref{thm::main} satisfies the canonical equations of the MP.
\end{theorem}
\begin{proof}
Recall that the MP requires us to find the \emph{lowest} $f_{ij}$'s whereas Theorem \ref{thm::main} dictates that we find the \emph{largest} $w_{ij}$'s. Let us define the terms $\tilde{w}_{ij} = -w_{ij}$. Thus, it is sufficient to show that $\tilde{w}_{ij} \leq \tilde{w}_{kl}$ implies that $f_{ij} \leq f_{kl}$. To do so, we need to perform a first-order Taylor expansion for $x(t)$ and $p(t)$ over the interval $[s=T-\Delta s, T]$, with $\Delta t>0$ small. Over this interval, the system is time-invariant. Let $A$ denote the system matrix over this interval. Then, from (\ref{eqn::xDyn}) and (\ref{eqn::xDyn}), we can write
\begin{eqnarray}
x(t) & = & e^{A(t-s)}x(s) \\
p(t) & = & 2 \int_t^T e^{-A(t-\tau)}(x(\tau)-\bar{x}) d\tau. \label{eqn::pDynNew}
\end{eqnarray}  
Let $P(t-s) := e^{A(t-s)} = I + (t-s)A + \bigO{\Delta s^2}$. We can then simplify re-write the above expressions as
\begin{eqnarray*}
x(t) & = & P(t-s)x(s)\\
& = & [I + (t-s)A]x(s) + \bigO{\Delta s^2},
\end{eqnarray*}
and
\begin{eqnarray*}
p(t) & = & 2\int_t^T P(\tau -t)[P(\tau-s)-M]x(s) d\tau\\
& = & 2\int_t^T [P(2\tau -t- s)-M]x(s) d\tau \\
& = & 2\int_t^T [I+(2\tau -t- s)A-M]x(s) d\tau + \bigO{\Delta s^2} \\
p(t) & = & [2(T-t)I + 2(T-t)(T-s)A-2(T-t)M]x(s) + \bigO{\Delta s^2}.
\end{eqnarray*}
Define $\xi(t,s) := t-s$ and write
\begin{eqnarray}
x(t) & = & \left[I + \xi(t,s)A\right]x(s) + \bigO{\Delta s^2}\\
p(t) & = & [2\xi(T,t)I + 2\xi(T,t)\xi(T,s)A-2\xi(T,t)M]x(s) + \bigO{\Delta s^2}.
\end{eqnarray}
Further, define the matrices 
\[
G := I + \xi(t,s)A, \quad H:=2\xi(T,t)I + 2\xi(T,t)\xi(T,s)A-2\xi(T,t)M.
\]
Ignoring the higher order terms $\bigO{\Delta s^2}$ for simplicity, we can write
\begin{eqnarray*}
\tilde{w}_{ij} & = & a_{ij}(x_i -x_j)(x_j-x_i) \\
& = & a_{ij}x(s)^T(g_i-g_j)(g_j-g_j)^Tx(s) \\
f_{ij} & = & a_{ij}x(s)^T(h_i-h_j)(g_j-g_j)^Tx(s),
\end{eqnarray*}
where $g_i^T$, $h_i^T$ are the $i$-th row of $G$ and $H$, respectively. Using the definitions of $G$ and $H$, we obtain
\begin{eqnarray*}
(g_i-g_j)(g_j-g_j)^T & = & -(I_i-I_j)(I_i-I_j)^T - \xi(t,s)[(I_i-I_j)(A_i-A_j)^T+(A_i-A_j)(I_i-I_j)^T]\\
&& - \xi(t,s)^2(A_i-A_j)(A_i-A_j)^T.
\end{eqnarray*}
The last term is quadratic so we can lump it into $\bigO{\Delta s^2}$. We then have
\begin{eqnarray*}
&& a_{ij}(g_i-g_j)(g_j-g_j)^T - a_{kl}(g_k-g_l)(g_l-g_k)^T  = (a_{kl}(I_k-I_l)(I_k-I_l)^T-a_{ij}(I_i-I_j)(I_i-I_j)^T) \\
&& + (a_{kl}(I_k-I_l)(A_k-A_l)^T-a_{ij}(I_i-I_j)(A_i-A_j)^T)\xi(t,s)\\
&& + (a_{kl}(A_k-A_l)(I_k-I_l)^T-a_{ij}(A_i-A_j)(I_i-I_j)^T)\xi(t,s) + \bigO{\Delta s^2}. 
\end{eqnarray*}
Similarly, we have
\begin{eqnarray*}
&&a_{ij}(h_i-h_j)(g_j-g_j)^T - a_{kl}(h_k-h_l)(g_l-g_k)^T  = (a_{kl}(I_k-I_l)(I_k-I_l)^T\\
&& -a_{ij}(I_i-I_j)(I_i-I_j)^T)2\xi(T,t)  + (a_{kl}(I_k-I_l)(A_k-A_l)^T-a_{ij}(I_i-I_j)(A_i-A_j)^T)2\xi(T,t)\xi(t,s)\\
&& + (a_{kl}(A_k-A_l)(I_k-I_l)^T-a_{ij}(A_i-A_j)(I_i-I_j)^T)2\xi(T,t)\xi(T,s) + \bigO{\Delta s^2}. 
\end{eqnarray*}
Let $\Gamma_1 = a_{kl}(I_k-I_l)(I_k-I_l)^T-a_{ij}(I_i-I_j)(I_i-I_j)^T$, $\Gamma_2 = a_{kl}(I_k-I_l)(A_k-A_l)^T-a_{ij}(I_i-I_j)(A_i-A_j)^T$, and $\Gamma_3 = a_{kl}(A_k-A_l)(I_k-I_l)^T-a_{ij}(A_i-A_j)(I_i-I_j)^T$. We now have
\begin{eqnarray*}
\tilde{w}_{ij}-\tilde{w}_{kl} & = & x(s)^T(\Gamma_1 + \xi(t,s) \Gamma_2 + \xi(t,s) \Gamma_2)x(s) + \bigO{\Delta s^2} \\
f_{ij}-f_{kl} & = & x(s)^T(2\xi(T,t)\Gamma_1 + 2\xi(T,t)\xi(t,s) \Gamma_2 + 2\xi(T,t)\xi(T,s) \Gamma_3)x(s) + \bigO{\Delta s^2}.
\end{eqnarray*}
But $\xi(T,t)\xi(t,s)$ and $\xi(T,t)\xi(T,s) $ are of order $\Delta s^2$ so we can also lump them into $\bigO{\Delta s^2}$ to obtain
\begin{equation}
f_{ij}-f_{kl} = x(s)^T(2\xi(T,t)\Gamma_1)x(s) + \bigO{\Delta s^2}.
\end{equation}
If $\tilde{w}_{ij}-\tilde{w}_{kl}\leq 0$, and since $\xi(T,t) \leq 0$, we can write
\begin{equation*}
2\xi(T,t)(\tilde{w}_{ij}-\tilde{w}_{kl}) = x(s)^T(2\xi(T,t)\Gamma_1 + 2\xi(T,t)\xi(t,s) \Gamma_2 + 2\xi(T,t)\xi(t,s) \Gamma_2)x(s) + \bigO{\Delta s^2} \leq 0,
\end{equation*}
or
\[
x(s)^T(2\xi(T,t)\Gamma_1)x(s) + \bigO{\Delta s^2} \leq 0,
\]
but the left hand side is $f_{ij}-f_{kl}$; hence, $\tilde{w}_{ij} \leq \tilde{w}_{kl} \implies f_{ij} \leq f_{kl}$ as required.

So far, we have verified the claim over the interval $[s,T]$ only. We need to verify that the claim holds over the interval $[r=T-2\Delta s,s]$. If the claim holds over this interval, then it can be generalized for the entire horizon of the problem $[0,T]$. The only complication that arises when studying this interval is that the terminal condition, i.e. $p(s)$, is not forced to be zero as in $[s,T]$. Let the system matrix over $[r,s]$ be $B$. Then, the state and costate are give by
\begin{eqnarray*}
x(t) & = &e^{B(t-r)}x(r) \\
p(t) & = & e^{-B(t-r)}p(r) -2\int_r^t e^{-B(t-\tau)}(x(\tau) -\bar{x}) d\tau. 
\end{eqnarray*}
Solving for $p(r)$ interns of $p(s)$ and substituting back, we can write $p(t)$ interns of $p(s)$ as follows:
\[
p(t) = e^{-B(t-s)}p(s) +2\int_t^s e^{-B(t-\tau)}(x(\tau) -\bar{x}) d\tau. 
\]
The integral term on the above expression is similar to that in (\ref{eqn::pDynNew}), and the same analysis above applies to it. We can obtain an expression for $p(s)$ from (\ref{eqn::pDynNew}) to arrive at the following solution over $[r,s]$:   
\begin{equation*}
p(t) = 2 \underbrace{\int_t^s e^{-B(t-\tau)}(x_B(\tau)-\bar{x}) d\tau}_{I_1} + 2 \underbrace{\int_s^T e^{-B(t-s)-A(s-\tau)}(x_A(\tau)-\bar{x}) d\tau}_{I_2}
\end{equation*}
Let $Q(t) = e^{Bt} = I+tB+\bigO{\Delta s^2}$. Then
\begin{eqnarray*}
I_1 & = & \left((s-t)I + (s-t)(s-r)B-(s-t)M\right)x(r) + \bigO{\Delta s^2} \\
I_2 & = & \int_s^T Q(s-t)P(\tau-s)[P(\tau-r)-M]x_A(s) d\tau, \\
& = &  \int_s^T [Q(s-t)P(2\tau-s-r)-M]x(s) d\tau\\
& = & \int_s^T [(I+(s-t)B)(I+(2\tau-s-r)A)-M]x_A(s) d\tau + \bigO{\Delta s^2}\\
& = & \int_s^T [I+(2\tau-s-r)A +(s-t)B-M]x_A(s) d\tau + \bigO{\Delta s^2} \\
& = & \left((T-s)I+\left(T^2-s^2-(s+r)(T-s)  \right)A +(s-t)(T-s)B-(T-s)M\right)x_A(s) + \bigO{\Delta s^2}.
\end{eqnarray*}
where we have used the fact $T-s = s- r = \Delta s$. Since the state is continuous, we have that $x_A(s) = x_B(s)$ or $x_A(s) = e^{B(s-r)}x(r)$. We can then write
\begin{eqnarray*}
I_2 & = & \left((T-s)I+\left(T^2-s^2-(s+r)(T-s)  \right)A +(s-t)(T-s)B-(T-s)M\right) \\
&& \cdot (I+(s-r)B)x(r) + \bigO{\Delta s^2} \\
& = & \left((T-s)I+\left(T^2-s^2-(s+r)(T-s)  \right)A +(T-s)(2s-t-r)B-(T-s)M\right)\\
&&\cdot x(r) + \bigO{\Delta s^2} 
\end{eqnarray*}
Summing both integrals, we obtain
\begin{eqnarray*}
p(t) & = & \left(2(T-t)I+ 2(T-s)(T-r)A + 2(T-s)(3s-2t-r)B - 2(T-t)M    \right)x(r)+ \bigO{\Delta s^2} \\ 
& = & \left(2\xi(T,t)I + 2\xi(T,s)\xi(T,r)A + 2\xi(T,s)(2\xi(s,t)+\xi(s,r)) B -2\xi(T,t)M \right)x(r) + \bigO{\Delta s^2}.
\end{eqnarray*}
Following similar steps to the above, we can derive the following expressions over the interval $[r,s]$: 
\begin{eqnarray*}
\tilde{w}_{ij}-\tilde{w}_{kl} & = & x(r)^T(\Gamma_1 + \xi(t,r) \Gamma_2 + \xi(t,s) \Gamma_2)x(r) + \bigO{\Delta s^2} \\
f_{ij}-f_{kl} & = & x(s)^T(2\xi(T,t)\Gamma_1 + 2\xi(T,t)\xi(t,r) \Gamma_2 + 2\xi(T,s)(2\xi(s,t)+\xi(s,r)) \Gamma_3 \\
&&+ 2\xi(T,s)(T,r)\Gamma_4)x(r)+ \bigO{\Delta s^2},
\end{eqnarray*}
where $\Gamma_1 = a_{kl}(I_k-I_l)(I_k-I_l)^T-a_{ij}(I_i-I_j)(I_i-I_j)^T$, $\Gamma_2 = a_{kl}(I_k-I_l)(B_k-B_l)^T-a_{ij}(I_i-I_j)(B_i-B_j)^T$, $\Gamma_3 = a_{kl}(B_k-B_l)(I_k-I_l)^T-a_{ij}(B_i-B_j)(I_i-I_j)^T$, and $\Gamma_4 = a_{kl}(A_k-A_l)(I_k-I_l)^T-a_{ij}(A_i-A_j)(I_i-I_j)^T$. Note again that $f_{ij}-f_{kl} $ can be simplified to
\begin{equation}
f_{ij}-f_{kl} = x(r)^T(2\xi(T,t)\Gamma_1)x(r) + \bigO{\Delta s^2}. \label{eqn::genDiffForm}
\end{equation}
Thus, by the same argument used over $[s,T]$, we conclude that $\tilde{w}_{ij}-\tilde{w}_{kl}  \leq 0$ implies that $f_{ij}-f_{kl} \leq 0$. From the structure of $p(t)$, and the above analysis, we conclude that $f_{ij}-f_{kl}$ will always have the same form as that in (\ref{eqn::genDiffForm}). Hence, we conclude that the claim holds over the entire horizon $[0,T]$.
\end{proof}
We now provide a geometric property of $u^\star(t)$.
\begin{lemma} \label{prob::ScaleInvar}
\emph{(Scale Invariance)} If $u^\star(t)$ is the optimal solution to (\ref{prb::JammerPrb}) with $x(0)=x_0$, then it is also optimal when starting from $x(0) = c\cdot x_0$, $c \in \mathbb{R}$.
\end{lemma}
\begin{proof}
Let $\tilde{x}(t)$ and $\tilde{p}(t)$ be state and co-state vectors corresponding to the initial state $\tilde{x}_0 = c \cdot x_0$. Then, by (\ref{eqn::xDyn}), (\ref{eqn::pDyn}), and uniqueness of the transition matrix, we have:
\begin{eqnarray*}
\tilde{x}(t) & = & \Phi_x(t,0)\tilde{x}_0 =  c \cdot x(t), \\
\tilde{p}(t) & = &  c \cdot 2\int_t^T \Phi_p(t,\tau)k(\tau)\left[\Phi_x(\tau,0)x_0-\bar{x}d\tau\right]  = c \cdot p(t).
\end{eqnarray*}
Hence, $\tilde{f}_{ij} = a_{ij}(\tilde{p}_j(t)-\tilde{p}_i(t))(\tilde{x}_i(t)-\tilde{x}_j(t)) = c^2 \cdot a_{ij}(p_j(t)-p_i(t))(x_i(t)-x_j(t))$. Hence, sgn$\left(f_{ij}\right) =$ sgn$\left(\tilde{f}_{ij}\right)$, $\forall i,j$.
\end{proof}
Consider the following sets for $i \in \{1,..., \sum_{j=0}^{\ell} {n \choose j} \}$: $\mathcal{S}_i= \left\{ x \in \mathbb{R}^n: u_i = \arg \max_{u\in \mathcal{U}} J(u), x(0) = x \right\}$. The set $\mathcal{S}_i$ corresponds to the set of initial conditions starting from which the solution to (\ref{prb::JammerPrb}) is stationary. In view of Lemma \ref{prob::ScaleInvar}, we conclude that the sets $S_i$ are linear cones.
\section{Attack II: An Adversary Capable Of Corrupting Measurements} \label{AttackII}
Assume now that the adversary is capable of adding a noise signal to all the nodes in the network in order to slow down convergence. The dynamics in this case are:
\begin{equation} \label{eqn::stateAttack2}
\dot{x}(t) = Ax(t) + u(t), \quad x(0) = x_0.
\end{equation}
We assume that the instantaneous power $u(t)^Tu(t) =: |u(t)|^2$ that the adversary can expend cannot exceed a fixed value $P_{max}$. We also assume that the adversary has sufficient energy $E_{max}$ to allow it to operate at maximum instantaneous power. Accordingly, the strategy space of the adversary is $\mathcal{U} = \left\{ u(t) \in C^1[0,T] :  |u(t)|^2 \leq P_{max} \right\},$ where $C^1[0,T]$ is a Banach space when endowed with the following norm: $\vnorm{x}_{C^1} = \vnorm{x}_{L_\infty} + \vnorm{\dot{x}}_{L_\infty}$. Thus, the adversary's problem is
\begin{eqnarray}
& \max_{u(t) \in \mathcal{U}} & \quad J(u) \label{prb::JammerPrb2} \\
& \text{s.t.} & \dot{x}(t) = Ax(t) + u(t),\quad A_{ii} =  - \sum_{j=1,j\neq i}^n A_{ij}, \nonumber \\
&& A_{ij} = A_{ji},  A_{ij} \geq 0, A_{ij} >0 \Leftrightarrow (i,j) \in \mathcal{E}. \nonumber
\end{eqnarray}
The Hamiltonian in this case is given by
\begin{eqnarray*}
H(x,p,u) & = &  k(t) \left|x(t) - \bar{x}\right|^2+ p(t)^T\left(Ax(t) + u(t)\right) \\
&& + \lambda(t)\left( |u(t)|^2 - P_{max}\right),
\end{eqnarray*} 
where $\lambda(t)$ is a continuously differentiable Lagrange multiplier associated with the power constraint. As before, we let $x(t),p(t) \in C^1[0,T]$. Here, $\lambda(t)$ must satisfy
\begin{equation*}
\lambda(t) \leq 0, \quad \lambda(t)\left(|u(t)|^2 - P_{max}\right) = 0.
\end{equation*} 
The first-order necessary conditions for optimality are:
\begin{eqnarray}
&& \dot{p}(t) = -\frac{\partial}{\partial x} H(x,u,p,t) \nonumber \\
&& \quad \quad = -2k(t)(x(t)-\bar{x}) - Ap(t), \quad p(T) = 0 \nonumber \\
&& \dot{x}(t) = Ax(t)+u(t), \quad x(0) = x_0  \nonumber \\
&& \frac{\partial}{\partial u}H(x,u,p,t) = 2\lambda(t)u(t) +p(t)=0. \label{eqn::ctrlAttack2}
\end{eqnarray}
To find $u^\star(t)$, consider the following cases:

\textbf{Case 1:} $\lambda(t) < 0 \implies  |u(t)|^2 = P_{max}$. Using (\ref{eqn::ctrlAttack2}), we obtain $\lambda(t) |u(t)|^2 = -\frac{1}{2}u(t)^Tp(t)$; hence,
\begin{equation}
\lambda(t) = -\frac{1}{2P_{max}}u(t)^Tp(t), \label{eqn::Lags}
\end{equation}
which we can then use to solve for the optimal control:
\[
u^\star(t) = P_{max}\frac{p(t)}{u(t)^Tp(t)} =\frac{E_{max}}{T}\cdot \frac{p(t)}{u(t)^Tp(t)}.
\]
\begin{remark} The optimal strategy $u^\star(t)$ is the vector of maximum power that it is aligned with $p(t)$. To see this, note that (\ref{eqn::Lags}) implies that $u(t)^Tp(t) > 0$, because $\lambda(t)<0$. Hence, the vectors $u^\star(t)$ and $p(t)$ are aligned. Define the unit vector $\bar{p}(t)= p(t)/\left|p\right|$. Then, we can further write
\begin{equation} \label{eqn:OptimalControl}
u^\star(t)=\frac{E_{max}/T}{\left|u\right|}\cdot \bar{p} = \sqrt{P_{max}}\cdot \bar{p}.
\end{equation}
Hence, the adversary's optimal solution in this case is to operate at the maximum power available. 
 \end{remark}
\textbf{Case 2:} $ |u(t)|^2 < P_{max} \implies \lambda(t) = 0$. Using (\ref{eqn::ctrlAttack2}), we obtain $p(t) = 0$. In this case the control is singular, since it does not appear in $\frac{\partial}{\partial u}H = 0$. But since $p(t)=0$ for all $t$, all its time derivatives must also be zero:
\begin{eqnarray}
&& \frac{d}{dt}\frac{\partial H}{\partial u} =  \dot{p}(t)  = -2k(t)\left(x(t)-\bar{x}\right) - A^Tp(t) = 0, \nonumber \\ && \therefore x(t) - \bar{x} =   0. \label{eqn::violate}
\end{eqnarray}
The conditions obtained by taking the time derivatives are also necessary conditions that must be satisfied at the optimal trajectory. However, (\ref{eqn::violate}) violates the initial condition. In order to resolve this inconsistency,  we set the control at $t=0$ to be an impulse, $u_i(t) = c\cdot \delta(t)$, in order to make $x(0)=\bar{x}$, where $c\in \mathbb{R}$ is chosen to guarantee $ |u(t)|^2 < P_{max}$. Note that we still have not recovered the control, and therefore we need to differentiate again
\begin{eqnarray}
&& \frac{d^2}{dt^2}\frac{\partial H}{\partial u} = \dot{x}(t) = Ax(t) + u(t) = 0, \nonumber \\
&& \therefore u(t) = -Ax(t) = -A\bar{x} = 0. \label{eqn::ctrlCase2Attack2}
\end{eqnarray}
\begin{remark}
Note that $x(t)=\bar{x}$ leads to having $u(t)=0$. This result matches intuition; when the nodes reach consensus, $J(u)=0$ for all $u(t) \in \mathcal{U}$. Hence, no matter what the control is, the utility of the adversary will always be zero. Thus, expending power becomes sub-optimal, and the optimal strategy is to do nothing. 
\end{remark}
Because the adversary attempts to increase the Euclidean distance between $x(t)$ and $\bar{x}$, we can readily see that $u(t)=0$ cannot be optimal, unless $x(t) = \bar{x}$. The following lemma proves this formally.
\begin{lemma} \label{lemma::equalityConst}
The solution of (\ref{prb::JammerPrb2}) satisfies $|u(t)|^2 = P_{max}$.
\end{lemma}
\begin{proof}
Assume that $|u_1(t)|^2< P_{max}$; then by (\ref{eqn::violate}) and (\ref{eqn::ctrlCase2Attack2}), $J(u_1) = 0$. Consider another solution which satisfies the power constraint with equality. Namely, let $u_2(t) = \sqrt{\frac{P_{max}}{n}}\mathbf{1}$. Using the solution to (\ref{eqn::stateAttack2}), and by defining the doubly stochastic matrix $P(t) = e^{At}$, for $t \geq 0$,
\begin{equation*}
x(t) = P(t) x_0 + \sqrt{\frac{P_{max}}{n}} \mathbf{1}t.
\end{equation*}
In this case, for $t\geq 0$, we have
\begin{eqnarray}%
&& \left|x(t) - \bar{x}\right|^2 = x_0^T(P(t)-M)^T(P(t)-M)x_0   \nonumber \\
&& + P_{max}t^2 + 2\sqrt{\frac{P_{max}}{n}}x_0^T(P(t)-M)\mathbf{1}t  \nonumber \\
&& = x_0^T(P(t)^2 - 2MP(t) - M^2)x_0 + P_{max}t^2  \label{eqn::step1}\\
&& = x_0^TP(2t)(I - M)x_0 + P_{max}t^2, \label{eqn::step2}
\end{eqnarray}
where (\ref{eqn::step1}) follows because $P(t)$, $t\geq 0$, and $M$ are stochastic matrices, and (\ref{eqn::step2}) follows from (\ref{prop::M}) and the semi-group property. Being a stochastic matrix, $P(2t)$ is positive semidefinite (psd). Also, $I-M$ is a Laplacian matrix; therefore, it is also psd. Further, note that
\begin{equation*}
P(2t)(I-M) = P(2t)-MP(2t) = (I-M)P(2t).
\end{equation*}
Hence, $P(2t)(I-M)$ is also psd, and therefore $x_0^TP(2t)(I - M)x_0 \geq 0$ for $t\geq 0$. This in turn implies
\begin{eqnarray*}
J(u_2) & = & \int_0^T k(t)\left[x_0^TP(2t)(I-M)x_0 + P_{max}t^2\right]dt \\
& \geq & \frac{P_{max}}{3}T^3 > J(u_1) = 0.
\end{eqnarray*}
We conclude that not utilizing the power budget available yields a lower utility for the adversary.\end{proof}
With Lemma \ref{lemma::equalityConst} at hand, it remains to determine the co-state vector in order to completely characterize $u^\star(t)$. To do so, we will invoke Banach's fixed-point theorem. To this end, we will work with the scaled utility $\tilde{J}(u) = \nu J(u)$, $\nu > 0$, without loss of generality. Note that $u^\star(t)$ in (\ref{eqn:OptimalControl}) is also the solution to the maximization problem of $\tilde{J}(u)$. The co-state trajectory is given by
\begin{equation} \label{eqn::costate}
p(t) = 2\nu \int_t^T k(\tau) P(\tau-t)(x(\tau)-\bar{x})d\tau.
\end{equation}
Substituting (\ref{eqn:OptimalControl}) and the solution to (\ref{eqn::stateAttack2}) into (\ref{eqn::costate}) yields
\begin{equation*}
p(t) = g(t) + 2\nu\sqrt{P_{max}}\int_t^T \int_0^\tau k(\tau)P(2\tau - (t+s)) \bar{p}(s) ds d\tau,
\end{equation*}
where $g(t) = 2\nu \int_t^T P(\tau-t)k(\tau) (P(\tau)x_0 - \bar{x}) d\tau$. Note that $2\tau - (t+s) \geq 0$ for $0 \leq s \leq \tau$, $t \leq \tau \leq T$, and hence $P(.)$ is a well-defined doubly stochastic matrix over the region of integration. We define the mapping $\mathcal{T}(p)(t):=p(t)$. By its structure, it is readily seen that $\mathcal{T}(p)(t): C^1[0,T] \to C^1[0,T]$. The following lemma aids in obtaining the co-state vector.
\begin{lemma} \label{lemma::normIneq}
Let $\tilde{\mathcal{T}}(x)(t) := k(t)\int_0^tP(s)x(s)ds$, where $P(t)$ is a doubly stochastic matrix, and fix $x(t) \in C^1[0,T]$. Then
\[
\vnorm{\tilde{\mathcal{T}}(x)}_{L_\infty} \leq \sup_{0\leq t \leq T} tk(t) \cdot \vnorm{x}_{L_\infty}.
\]
\end{lemma}
\begin{proof}
We have:
\begin{eqnarray*}
&&\vnorm{\tilde{\mathcal{T}}(x)}_{L_\infty}  =  \sup_{0\leq t \leq T} \vnorm{k(t)\int_0^tP(s)x(s)ds}_{L_\infty}\\
&&= \sup_{0\leq t \leq T} k(t) \sup_{1\leq i \leq n} \left| \int_0^t \sum_{j=1}^n P_{ij}(s)x_j(s)ds \right| \\
&& \leq  \sup_{0\leq t \leq T} k(t) \sup_{1\leq i \leq n}  \int_0^t \sum_{j=1}^n P_{ij}(s)\left| x_j(s)\right|ds
\end{eqnarray*}
\begin{eqnarray*}
&&\overset{(a)}{\leq}  \sup_{0\leq t \leq T} k(t) \sup_{1\leq i \leq n}  \int_0^t \left(\sum_{j=1}^n P_{ij}(s)\right)\sup_{1\leq j \leq n}\left| x_j(s)\right|ds\\
&&=  \sup_{0\leq t \leq T} k(t) \int_0^t \sup_{1\leq j \leq n}\left| x_j(s)\right|ds\\
&& \leq  \sup_{0\leq t \leq T} k(t) \int_0^t \sup_{0\leq s \leq T} \sup_{1\leq j \leq n}\left| x_j(s)\right|ds\\
&&= \sup_{0\leq t \leq T} tk(t) \cdot\vnorm{x}_{L_\infty},
\end{eqnarray*}
where $(a)$ follows from H\"{o}lder's inequality.  
\end{proof}
\begin{theorem}
By choosing $\nu < \frac{1}{2\sqrt{P_{max}}(\check{k}+\hat{k})}$, where $\check{k} = \sup_{0\leq t \leq T} tk(t)$ and $\hat{k} = \sup_{0\leq t \leq T}\int_t^T \tau k(\tau)d\tau$, the mapping $\mathcal{T}(p)(t): C^1[0,T] \to C^1[0,T]$ has a unique fixed point $p^\star(t) \in C^1[0,T]$ that can be obtained by any sequence generated by the iteration $p_{k+1}(t) = \mathcal{T}(p_k)(t)$, starting from an arbitrary vector $p_0(t) \in C^1[0,T]$.
\end{theorem}
\begin{proof}
The theorem will follow if for this choice of $\nu$, the mapping $\mathcal{T}$ is a contraction. Consider two vectors $y(t),z(t) \in C^1[0,T]$ and let $\bar{y}(t),\bar{z}(t)$ be the corresponding normalized unit norm vectors. Let $\bar{w} = \bar{y}-\bar{z}$. Then
\begin{eqnarray*}
&&\hspace{-3mm}\frac{1}{2\nu\sqrt{P_{max}}}\vnorm{\mathcal{T}(y) - \mathcal{T}(z) }_{C^1} = \\
&&\hspace{-3mm}\sup_{0\leq t \leq T}k(t) \sup_{1\leq i \leq n} \left|\int_0^t\sum_{j=1}^nP_{ij}(t-s)\bar{w}_j(s)ds \right| \\
&&\hspace{-3mm}+ \sup_{0\leq t \leq T}  \sup_{1\leq i \leq n} \left| \int_t^Tk(\tau) \int_0^\tau \sum_{j=1}^nP_{ij}(2\tau - (t+s)) \bar{w}_j(s) ds d\tau\right| \\
&&\hspace{-3mm}\leq \sup_{0\leq t \leq T}tk(t) \vnorm{\bar{w}}_{L_\infty} + \sup_{0\leq t \leq T}  \sup_{1\leq i \leq n} \int_t^Tk(\tau) \\
&&\hspace{-3mm}\cdot \int_0^\tau \sum_{j=1}^nP_{ij}(2\tau - (t+s)) \left| \bar{w}_j(s)\right| ds d\tau ,
\end{eqnarray*}
where the last inequality follows from Lemma \ref{lemma::normIneq}. Using arguments similar to those used in proving Lemma \ref{lemma::normIneq}, we have:
\begin{align}
&\frac{1}{2\nu\sqrt{P_{max}}}\vnorm{\mathcal{T}(y) - \mathcal{T}(z) }_{C^1}\\\nonumber &\leq  \left(\sup_{0\leq t \leq T}tk(t) + \sup_{0\leq t \leq T} \int_t^T\tau k(\tau) d\tau\right)\vnorm{\bar{w}}_{L_\infty} \cr
&\leq (\check{k}+\hat{k})\vnorm{y - z}_{L_\infty} \leq 2\nu\sqrt{P_{max}}(\check{k}+\hat{k})\vnorm{y - z}_{C^1},
\end{align}
where the second inequality follows from the properties of similar triangles. We readily see that by selecting $\nu < \frac{1}{2\sqrt{P_{max}} (\check{k}+\hat{k})}$, the last inequality implies that $\mathcal{T}(p)(t)$ is a contraction mapping. Because $C^1[0,T]$ endowed with $\vnorm{.}_{C^1}$ is a Banach space, Banach's contraction principle guarantees the existence of a unique fixed point $p^\star(t) \in C^1[0,T]$ which can be obtained from the iteration $p_{k+1}(t) = \mathcal{T}(p_k)(t)$ as $k\to \infty$, for any initial point.
\end{proof}
\section{Numerical Results} \label{Simulations}
In this section, we provide a numerical example for \textsc{Attack-I}. We consider the complete graph with $n=4$. The matrix $A(0)$ is generated at random and is equal to
\[
A(0)=\left(\begin{array}{cccc}-2.1293 & 0.0326 & 0.5525 & 1.5442 \\ 0.0326 & -1.2191 & 1.1006 & 0.0859 \\ 0.5525 &  1.1006& -3.1447 & 1.4916 \\1.5442 & 0.0859 & 1.4916 & -3.1217\end{array}\right)
\]
We fix $\ell=2$, $T=2$, and $x_0 = [1,2,3,4]^T$ -- hence, $x_{avg} = 2.5$. We simulated the network using \textsc{Matlab}'s \textsc{Bvp Solver} and computed the optimal control using (\ref{eqn::OptCtrlMP}), which was found to be $u^\star(t) = [1,0,1,0,1,1]^T$ for $t \in [0,2]$. Indeed, at $t=0$, the highest $w_{ij}$ values are $w_{13}(0)= 2.2101$ and $w_{14}(0)=13.8979$ which confirms the conclusion of Thm \ref{thm::main}. In this particular example, $w_{13},w_{1,4}$ remain dominant throughout the problem's horizon, and hence the control is stationary. Fig. \ref{fig::advNoadv} simulates the network at hand with and without the presence of the adversary. Note that the adversary was successful in delaying convergence. Since both links the adversary broke emanate from node $1$, $x_1(t)$ is far from consensus.
\begin{figure}[!t]
\centering
\includegraphics[width=9.5cm]{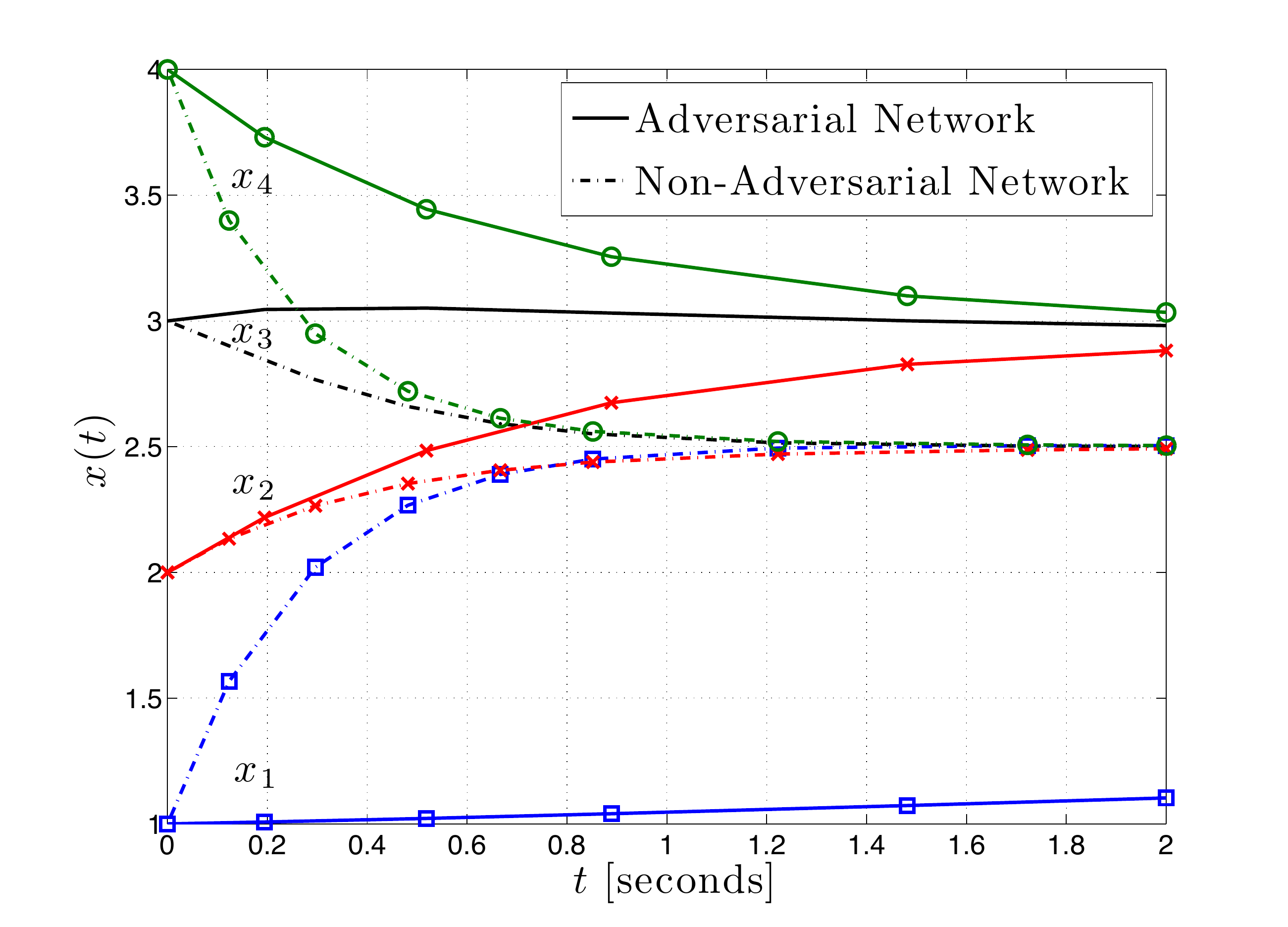}
\caption{Effect of \textsc{Attack-I} on the convergence to consensus. $T=2$, $n=4$, $\ell=2$, and $x_0 = [1,2,3,4]$.}
\label{fig::advNoadv}
\end{figure}
\section{Conclusion} \label{Conclusion}
We have considered two types of adversarial attacks on a network of agents performing consensus averaging. Both attacks have the common objective of slowing down the convergence of the nodes to the global average. \textsc{Attack-I} involves an adversary that is capable of compromising links, with a constraint on the number of links it can break. Despite the interdependence of the state, co-state, and control, we were able to find the optimal strategy. We also presented a potential-theoretic interpretation of the solution. In \textsc{Attack-II}, a finite power adversary attempts to corrupt the values of the nodes by injecting a signal of bounded power. We assumed that the adversary has sufficient energy $E_{max}$ to operate at maximum instantaneous power and derived the corresponding optimal strategy. It would be interesting to consider the case when $E_{max} < T\cdot P_{max}$, when the adversary cannot expend $P_{max}$ at each time instant. This will be explored in future work. 
\bibliographystyle{IEEEtran}
\bibliography{references}


\end{document}